 \documentclass[11pt,letterpaper]{amsart}

\usepackage{graphicx}
\makeatother

\theoremstyle{definition}

\def\fnum{equation} 
\newtheorem{Thm}[\fnum]{Theorem}
\newtheorem{Cor}[\fnum]{Corollary}

\newtheorem{Lem}[\fnum]{Lemma}

\numberwithin{equation}{section}

\newcommand{\osc}{{\text{osc}}}

\def\RR{{\mathbb R}}

\def\CC{{\mathbb C }}

\newcommand{\e}{{\text {e}}}

\newcommand{\eqr}[1]{(\ref{#1})}

\title{Liouville theorem for immersed minimal surfaces in any codimension}

\author{Tobias Holck Colding}%
\address{MIT, Dept. of Math.\\
77 Massachusetts Avenue, Cambridge, MA 02139-4307.}
\author{William P. Minicozzi II}%

\thanks{The  authors
were partially supported by NSF  DMS Grants   2405393 and 2304684.}

\email{colding@math.mit.edu and minicozz@math.mit.edu}

\begin{document}

\maketitle

\begin{abstract}    
For a proper immersed minimal disk in $\RR^N$ with quadratic area growth, we 
show that any harmonic function  whose negative part grows at a slow sub-linear rate   is constant.  This leads to a higher codimensional Bernstein theorem for minimal disks contained in a sub-linearly growing cone.  The catenoid, helicoid and Enneper's family of surfaces together show that this result is optimal.  We also show uniform H\"older regularity of harmonic functions.  
\end{abstract}


\section{Introduction}

Throughout $\Sigma^2\subset \RR^N$ is a properly immersed minimal surface that is topologically a disk.  
We will use $|S|$ and $|\sigma|$ to denote the area of a two-dimensional set $S$ and the length of a one-dimensional set $\sigma$, respectively.  

\vskip2mm
Our Liouville theorem is the following:

\begin{Thm}   \label{t:Liouville0}
Suppose that $\Sigma$ is a properly immersed minimal surface in $\RR^N$ that is topologically a disk with $|B_r\cap \Sigma|\leq C_a\,r^2$.
If $u$ is  a  harmonic function on $\Sigma$ with 
\begin{align}
	  -C\,(1+|x|^{\alpha}) \leq u(x)  {\text{ for some   $\alpha<-\frac{\log (1- \e^{-24\, C_a})}{\log 2}$ and any constant $C$}}\,  , \notag
\end{align}
then $u$ is constant.  
\end{Thm}

Since the coordinate functions are harmonic on $\Sigma$, we must   have $\alpha < 1$ for a Liouville theorem. 
Enneper's surface, described below, shows that $\alpha < \frac{1}{3}$ when $C_a = 3\, \pi$.  As $k \to \infty$, the  Enneper surface
of order $k$
has $\alpha(k) \to 0$ and $C_a (k)\to \infty$, see subsection $2.1.3$ in \cite{K}.
The catenoid shows that, even in $\RR^3$, any $\alpha>0$ would not suffice to rule out nonconstant harmonic functions  without a topological assumption.  Similarly, the helicoid has cubic area growth and a logarithmically growing harmonic function, showing
 that the quadratic area bound is necessary.

\vskip1mm
Combining Theorem \ref{t:Liouville0}
with the volume growth estimates of \cite{CM9} gives a Bernstein theorem for minimal disks with slowly growing height in any codimension:

\begin{Cor}   \label{c:Liouville}
Given $N$, there exists $\alpha > 0$ so that if 
  $\Sigma$ is a properly immersed minimal disk in $\RR^N$ that is   contained in $\{x\in \RR^N\,|\,\sum_{i=3}^N|x_i|\leq C\,(|x_1|^{\alpha}+|x_2|^{\alpha}+1)\}$ for   some constant $C$, then $\Sigma$ is the $x_1$--$x_2$ plane.  
\end{Cor}

Broadly speaking, there are two types of Liouville theorems.  Both assert that certain harmonic functions must be constant.  The first type is for bounded or even positive harmonic functions; see, e.g. \cite{Ya}.  The other is more general and allows for a  slowly growing negative lower bound; see, e.g. \cite{CgYa,CM8}.
It is a  classical fact that on any manifold with quadratic volume growth any positive harmonic function is constant.  

 The Hoffman-Meeks half-space theorem can also be thought of as a Liouville theorem albeit for the special harmonic functions that are the coordinate functions.  It states  that if  $\Sigma$ is properly immersed in a half-space in $\RR^3$, then it must be a plane, \cite{HM} (cf. also \cite{CM10}).
This is a   Liouville theorem for the positive harmonic function given by the  distance to the boundary of the half-space.  Without the assumption of proper the Liouville theorem fails in general.  Indeed Jorge-Xavier constructed non-flat minimal surfaces in a slab between two planes,  \cite{JX}.
Nadirashvili showed that  a minimal surface can even be contained in a ball, \cite{N}.
Without any topological assumption, the example of a catenoid in $\RR^3$, where the vertical coordinate function grows logarithmically,
 shows that  the conclusion of Theorem \ref{t:Liouville0} can fail.   By considering the universal cover of the catenoid,
the Liouville theorem does not hold if   the immersion is not proper.

 The classical Enneper surface{\footnote{Enneper's surface is parametrized conformally  (see, e.g., page $13$ in \cite{CM1}) by $$(u,v) \to \left(u - \frac{u^3}{3} + u \, v^2, - v + \frac{v^3}{3} - v\, u^2 , u^2 - v^2\right)\, . $$
 The functions $u$ and $v$ on the plane correspond to harmonic functions on Enneper's surface that grow
 polynomial at the rate $\frac{1}{3}$.     }} is a properly immersed minimal disk.  It has quadratic area growth with
  $C_a = 3\, \pi$, the tangent cone at infinity is a plane with multiplicity three, and the vertical coordinate function is a nonconstant harmonic function that grows polynomially at the rate $\frac{2}{3}$.   In particular, for any $\alpha > \frac{2}{3}$, there is a constant $C$ so that Enneper satisfies
  \begin{align}
  	|x_3| \leq C \, (|x_1|^{\alpha} + |x_2|^{\alpha} +1) \, .
  \end{align}
    There are   polynomially growing harmonic functions on Enneper's surface growing at the rate of $\frac{1}{3}$, but these are not seen geometrically.
    
\begin{figure}[ht]
    \centering
    \includegraphics[width=0.9\textwidth]{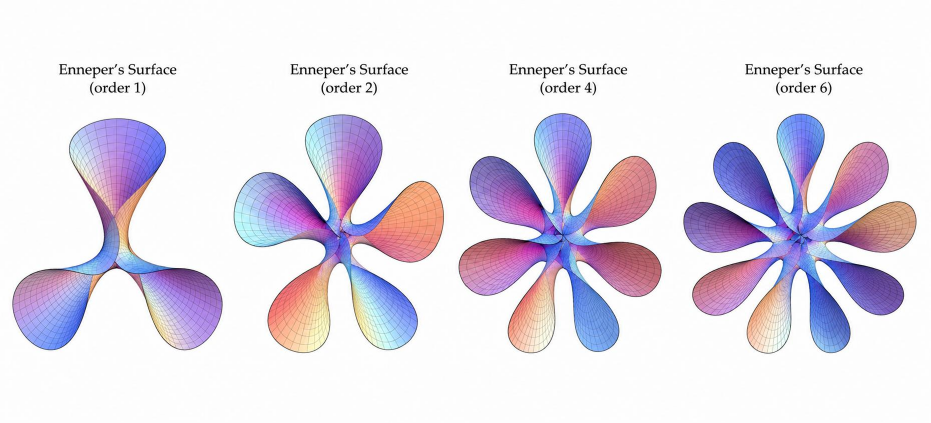}
    \caption{Enneper surfaces of orders 1, 2, 4, and 6}
    \label{fig:enneper}
\end{figure}

Without quadratic area growth,  there is no such Liouville theorem no matter how slow the sub-linear growth is.  The helicoid{\footnote{The
   helicoid is   parametrized conformally by (see, e.g., page $8$ in \cite{CM1})     $$
   	(u,v) \to (\cosh u \, \cos v , \cosh u \, \sin v , v ) \, .
   $$
   There is a logarithmically growing harmonic function  on the helicoid that corresponds to the coordinate function $u$ on the plane.}} has cubic area growth and has a nonconstant harmonic function that grows logarithmically, and thus more slowly than any polynomial rate. The helicoid is conformal to the plane, so any bounded (or positive) harmonic function is constant.

   We will also prove uniform H\"older regularity for harmonic functions  which can be viewed as an effective version of the Liouville theorem:

\begin{Thm}  \label{t:Liouville1.0}
Suppose that $\Sigma\subset B_{2\,r}\subset \RR^N$ is a compact immersed minimal surface with $\partial \Sigma\subset \partial B_{2\,r}$ that is topologically a disk with $|\Sigma|\leq 4\,C_a\,r^2$.
There exists $\alpha=\alpha (C_a)>0$ such that if $u$ is a harmonic function on $\Sigma$, $x$, $y$ lie in the same connected component of $B_s\cap \Sigma$, where $0<s<r$, then
\begin{align}
|u(x)-u(y)|\leq C\,\|u\|_{L^1(\Sigma)}\,\left(\frac{s}{r}\right)^{\alpha}\, . \notag
\end{align}
\end{Thm}

The examples above show that the smallest possible lower bound that would give a Liouville theorem is in terms of the distance to some power $<1$ that depends on the area growth.  This corresponds to that the exponent in the H\"older regularity is $<1$ and depends on the area.  If  the minimal surface is embedded in $\RR^3$ and the harmonic function is a coordinate function, then more is known.  Indeed the one-sided curvature estimate in \cite{CM2,CM3,CM4,CM5} can be interpreted as showing a gradient estimate for the log of a coordinate function on embedded minimal disks in $\RR^3$ without  priori  bounds on area or curvature.    See \cite{BB,CM6,CM7,MP,MPR,MT,P} for applications and more results in this direction.

 The class of minimal surfaces with quadratic area growth is very rich and includes both classical minimal surfaces of finite total curvature in $\RR^3$, \cite{CM1,MP,Os},
 as well as complex algebraic curves  in dimensions four and above, see, e.g. \cite{J}.

 \vskip2mm
A point of interest is that the Gauss-Bonnet theorem is never used in the proofs.  Instead almost all techniques are applicable in higher dimensions as well.  
The catenoid shows that some topological assumption is needed in Theorem \ref{t:Liouville1.0}.

\section{Bound for the oscillation}

Fix $x_0\in \Sigma$. For $r>0$, we let $\Sigma_r$ be the connected component of $B_r (x_0)\cap \Sigma$ containing $x_0$.  Observe  that if $r<s$, then $\Sigma_r\subset \Sigma_s$.   
If $u$ is a bounded function on $\Sigma_r$, then the oscillation $\osc_r^0 u$ of $u$ on $\Sigma_r$ is defined to be the amount that $u$ goes below $u(x_0)$ on $\Sigma_r$ 
\begin{align}
\osc_r^0 u= u(x_0) -\inf_{\Sigma_r}u\,  .
\end{align}
We assume that 
\begin{align}
|\Sigma_r|\leq C_a\,r^2\,  .
\end{align}

The next growth estimate for the oscillation is the key to both the Liouville theorem:

\begin{Lem}  \label{l:osc}
If $u$ is a harmonic function on $\Sigma_{2\,r}$, then
\begin{align}
\osc_r^0 u\leq \gamma\,\osc_{2\,r}^0 u\,  ,
\end{align}
where $\gamma= 1- \e^{-24\, C_a}<1$.  
\end{Lem}

\begin{proof}
We can assume that $u$ is nonconstant as else there is nothing to show.  It follows from the strong maximum principle that $\osc^0_r u>0$ for $r>0$.  After replacing $u$ by 
\begin{align}
 \frac{u-\inf_{\Sigma_{2\,r}}u}{ \osc^0_{2\,r}u}\,  ,
\end{align}
we may assume that $ u>0$ on $\Sigma_{2\,r}$, $u(x_0) = 1 $, and $\osc_{2r}^0 u = 1$.

Since $u>0$ is harmonic, then $v=-\log u$ satisfies 
\begin{align}
	\Delta\,v=|\nabla v|^2 \, .
\end{align}
     Let $\phi$ be a cutoff function that is $1$ on $B_{\frac{3}{2}\,r}$, $|\nabla \phi|\leq \frac{2}{r}$ and has support in $B_{2\,r}$.  Using that $\Delta\,v=|\nabla v|^2$, integration by parts together with the Cauchy-Schwarz inequality gives  
\begin{align}
\int_{\Sigma_{2\,r}} \phi^2\,|\nabla v|^2=\int_{\Sigma_{2\,r}} \phi^2\,\Delta\,v=-2\int_{\Sigma_{2\,r}} \phi\,\langle\nabla \phi,\nabla v\rangle\leq \frac{1}{2}\int_{\Sigma_{2\,r}} |\nabla v|^2\,\phi^2+2\int_{\Sigma_{2\,r}} |\nabla \phi|^2\,  .
\end{align}
It follows that 
\begin{align}  \label{e:Liouv1}
\int_{\Sigma_{\frac{3}{2}\,r}}|\nabla v|^2\leq \int_{\Sigma_{2\,r}} \phi^2\,|\nabla v|^2\leq 4\int_{\Sigma_{2\,r}} |\nabla \phi|^2\leq 16\,r^{-2}\,|\Sigma_{2\,r}|\,  .
\end{align}

We will next bound $M=\sup_{\Sigma_r}v$ in terms of $\int_{\Sigma_{\frac{3\,r}{2}}}|\nabla v|$.  
  Since  $v(x_0) = 0$ and $\Sigma_r$ is connected, we have that
\begin{enumerate}
\item[($\dagger$)] $v$ takes every value in $[0 , M)$ on $\Sigma_r$.
\end{enumerate}
Suppose that $s \in (0,M)$ is a regular value of $v$ on $\Sigma$ and set $\sigma_s = \Sigma_{\frac{3r}{2}} \cap \{ v = s \}$.  We claim that any component 
$\tilde{\sigma}_s$ of $\sigma_s$ must  go to the boundary $\partial \Sigma_{ \frac{3r}{2}}$.  If this was not the case, then we would have a simple closed curve in the topological disk $\Sigma$  where $u$ was constant.  This would be forced to bound a disk $\tilde{\Sigma}$  in $\Sigma$ which, by the convex hull property (see, for instance, proposition 1.9 in \cite{CM1}), must be contained in 
$\Sigma_{ \frac{3r}{2}}$.  The maximum principle would then  imply that $u$ is constant on $\tilde{\Sigma}$, which forces $u$ to be constant everywhere by unique continuation.  Since $u$ is assumed to be nonconstant, this gives the contradiction.

Property ($\dagger$) implies that at least one component $\tilde{\sigma}_s$ intersects $\Sigma_r$.  Since $\tilde{\sigma}_s$ also goes to  $\partial B_{ \frac{3r}{2}}$, we conclude that
\begin{align}
	 |\sigma |\geq \frac{r}{2} \, .
\end{align}
Since almost every value is regular by Sard's theorem (Theorem $3.4.3$ in \cite{Fe}), 
   the coarea formula  (page $243$ in \cite{Fe})   gives that
\begin{align}
\int_{\Sigma_{\frac{3\,r}{2}}}|\nabla v|&=\int_0^{\infty} \left| \{v=s\}\cap \Sigma_{\frac{3\,r}{2}} \right| \, ds
	\geq \int_0^{M}\left| \{v=s\}\cap \Sigma_{\frac{3\,r}{2}} \right| \, ds \geq \frac{r}{2}\,M\,  .
\end{align}
Combining this with  the Cauchy-Schwarz inequality gives that 
\begin{align}
\frac{r^2}{4}\,M^2 \leq \left( \int_{\Sigma_{\frac{3\,r}{2}}}|\nabla v| \right)^2 \leq \left| \Sigma_{\frac{3\,r}{2}} \right| \, \int_{\Sigma_{\frac{3\,r}{2}}}|\nabla v|^2
\leq  C_a \, \frac{ 9\, r^2}{4} \,\int_{\Sigma_{\frac{3\,r}{2}}}|\nabla v|^2 \, .
\end{align}
Using \eqr{e:Liouv1} to bound the energy term now gives
\begin{align}
	M^2 \leq 9\, C_a\, \int_{\Sigma_{\frac{3\,r}{2}}}|\nabla v|^2 \leq 9\, C_a \, \frac{ 16 \, |\Sigma_{2r}|}{r^2} \leq (24  \, C_a)^2 \, , 
\end{align}
so that $M \leq 24 \, C_a$.
This gives that on $\Sigma_r$
\begin{align}
v\leq M \leq 24 \, C_a\, .  
\end{align}
It follows  that 
$\e^{-24\, C_a}\leq \e^{-M}\leq u$ on $\Sigma_r$ and therefore 
\begin{align}
	\osc_r^0 u\leq 1-\e^{-24\, C_a}=\gamma\,\osc_{2\,r}^0 u \, .
\end{align}
\end{proof}

\subsection{Liouville theorem}
We are now ready to prove the Liouville theorem.

\begin{proof}
[Proof of Theorem \ref{t:Liouville0}] 
If $u$ is not constant, then   the strong maximum principle and unique continuation imply that $\osc_{1}^0u >0$.  For each positive integer $k$, set 
\begin{align}
a_k=\osc_{2^k}^0 u\,  .
\end{align}
By Lemma \ref{l:osc}, we get for every $k$ that
\begin{align}
a_{k+1}\geq \gamma^{-1}\,a_k\, ,
\end{align}
and therefore
\begin{align}
a_{k}\geq \gamma^{1-k}\,a_1\,  .
\end{align}
By assumption, $u(x) \geq -C\,(1+|x|^{\alpha})$ for some $C$ and some $\alpha > 0$.  After increasing $C$, we get for all $r \geq 1$ that 
  $ \osc_r^0 u  \leq  C\,r^{\alpha}$ for the same $\alpha$, and thus 
\begin{align}
a_k\leq  C\,2^{k\,\alpha}\,  .
\end{align}
It follows that $\gamma^{1-k}\,a_1\leq  C\,2^{k\,\alpha}$ or, equivalently, 
\begin{align}
\frac{a_1\, \gamma}{ C}\leq (\gamma\,2^{\alpha})^k\,  .
\end{align}
However,  $\gamma\,2^{\alpha}<1$ by assumption, so the right-hand side goes to zero as $k \to \infty$. This gives the desired contradiction.  
\end{proof}

\begin{proof}
[Proof of Corollary \ref{c:Liouville}]
By \cite{CM9},
 $\Sigma$ has quadratic area growth.  We can therefore apply Theorem \ref{t:Liouville0} to each  coordinate functions $x_i$ for $i=3,\cdots,N$ to conclude that each of these is constant.  The corollary  follows.  
\end{proof}

In another direction, Schoen and Simon, \cite{SS}, proved a Bernstein theorem for proper embedded minimal disks in $\RR^3$ with quadratic area growth;
  cf. corollaries $1.7$ and $1.18$ in \cite{CM3} and the stronger results in
  \cite{CM5} without any a priori bounds.

\subsection{H\"older regularity of harmonic functions}

To prove Theorem \ref{t:Liouville1.0}, suppose that $x\in \Sigma$ and $r>0$.  Let $\Sigma_{x,r}$ be the connected component of $B_r(x)\cap\Sigma$ containing $x$.   Set 
\begin{align}
\osc_r^xu=\sup_{\Sigma_{x,r}}u-\inf_{\Sigma_{x,r}}u\,  .
\end{align}

The key is the following:

\begin{Lem}  \label{l:oscv2}
There exists $\gamma<1$, such that if $u$ is a harmonic function on $\Sigma_{2\,r}$, then
\begin{align}
\osc_r^xu\leq \gamma\,\osc^x_{2\,r}u\,  .
\end{align}
\end{Lem}

\begin{proof}
This follows by applying Lemma \ref{l:osc} to both $u$ and $-u$.
\end{proof}

Using this, we can prove the H\"older regularity of harmonic functions.

\begin{proof}[Proof of Theorem \ref{t:Liouville1.0}]
Suppose that $x$, $y\in \Sigma_s$. Applying Lemma \ref{l:oscv2} iteratively (similarly to how Lemma \ref{l:osc} was applied iteratively in the proof of Theorem \ref{t:Liouville0}), we get 
\begin{align}
|u(x)-u(y)|\leq \left(\frac{s}{r}\right)^{-\frac{\log \gamma}{\log 2}}\,\osc^x_{r}u\,  .
\end{align}
Since $|u|=\max\{u,-u\}$ is sub-harmonic, it follows from the mean-value inequality, see, e.g., Corollary $1.17$ in \cite{CM1}, that 
\begin{align}
\osc^x_{r}u\leq 2\,\sup_{\Sigma_{r}}|u|\leq C_n\,r^{-2}\int_{\Sigma_{2\,r}}|u|\,  .
\end{align}
Combining these two inequalities gives the claim.  
\end{proof}


\begin{thebibliography}{A}  
   
 \bibitem[BB]{BB}
 J. Bernstein and C. Breiner,  {\it Conformal structure of minimal surfaces with finite topology},
  Comment. Math. Helv. 86 (2011), no. 2, 353--381.
  
 \bibitem[CgYa]{CgYa} 
 S.Y. Cheng and S.T. Yau, \emph{Differential equations on Riemannian manifolds and their geometric applications}. Comm. Pure Appl. Math. 28 (1975), no. 3, 333-354.

 \bibitem[CM1]{CM1} T.H. Colding and W.P. Minicozzi II, 
\emph{A course in minimal surfaces}, Graduate Studies in Mathematics, vol. 121 (American Mathematical Society, Providence, RI, 2011). 

\bibitem[CM2]{CM2} T.H. Colding and W.P. Minicozzi II, 
{\it The space of embedded minimal surfaces of fixed genus in a 3-manifold. I. Estimates off the axis for disks},
 Ann. of Math. (2) 160 (2004), no. 1, 27--68.

\bibitem[CM3]{CM3} T.H. Colding and W.P. Minicozzi II, 
{\it The space of embedded minimal surfaces of fixed genus in a 3-manifold. II. Multi-valued graphs in disks},
 Ann. of Math. (2) 160 (2004), no. 1, 69--92.

\bibitem[CM4]{CM4} T.H. Colding and W.P. Minicozzi II, 
{\it The space of embedded minimal surfaces of fixed genus in a 3-manifold. III. Planar domains},
 Ann. of Math. (2) 160 (2004), no. 2, 523--572.

\bibitem[CM5]{CM5} T.H. Colding and W.P. Minicozzi II, 
{\it The space of embedded minimal surfaces of fixed genus in a 3-manifold. IV. Locally simply connected},
Ann. of Math. (2) 160 (2004), no. 2, 573--615.

\bibitem[CM6]{CM6} T.H. Colding and W.P. Minicozzi II, 
{\it The space of embedded minimal surfaces of fixed genus in a 3-manifold V; fixed genus},
 Ann. of Math. (2) 181 (2015), no. 1, 1--153.

\bibitem[CM7]{CM7} 
T.H. Colding and W.P.  Minicozzi II,  
{\it The Calabi-Yau conjectures for embedded surfaces},
 Ann. of Math. (2) 167 (2008), no. 1, 211--243. 
 
\bibitem[CM8]{CM8} 
T.H. Colding and W.P.  Minicozzi II,  
\emph{Liouville properties}. ICCM Not. 7 (2019), no. 1, 16-26.

\bibitem[CM9]{CM9} 
T.H. Colding and W.P.  Minicozzi II,  
\emph{Minimal submanifolds confined in space}, preprint.

\bibitem[CM10]{CM10} 
T.H. Colding and W.P.  Minicozzi II,  
\emph{Distance between minimal surfaces and  flows}, preprint.

 \bibitem[Fe]{Fe}
 H. Federer,  {\it Geometric measure theory}. Die Grundlehren der mathematischen Wissenschaften, 
 Band 153 Springer-Verlag New York, Inc., New York 1969.
 
 \bibitem[HM]{HM}  D. Hoffman and W.H. Meeks III, 
\emph{The strong halfspace theorem for minimal surfaces}.  
Invent. Math. 101 (1990), no. 2, 373--377.

\bibitem[J]{J}
P.W. Jones,  
{\it A complete bounded complex submanifold of $\CC^3$}, 
Proc. Amer. Math. Soc. 76 (1979), no. 2, 305--306.

 \bibitem[JX]{JX}
L. Jorge and F. Xavier,  
{\it A complete minimal surface in $\RR^3$ between two parallel planes},
Ann. of Math. (2) 112 (1980), no. 1, 203--206. 

\bibitem[K]{K}
H. Karcher, {\it Construction of minimal surfaces}. Surveys in Geometry, 1--96, 1989.
University of Tokyo, 1989, and Lecture Notes No. 12, SFB 256, Bonn, 1989, 
http://www.math.uni-bonn.de/people/karcher/karcherTokyo.pdf.


\bibitem[MM]{MM}
F. Martin and W.H.  Meeks III,  {\it Calabi-Yau domains in three manifolds},
 Amer. J. Math. 134 (2012), no. 5, 1329--1344.


\bibitem[MP]{MP}
W.H.  Meeks III  and J. P\'erez, 
{\it The classical theory of minimal surfaces},
 Bull. Amer. Math. Soc. (N.S.) 48 (2011), no. 3, 325--407.
 
 \bibitem[MPR]{MPR}
W.H.  Meeks III,   J. P\'erez and A. Ros, 
{\it The embedded Calabi-Yau conjecture for finite genus},
 Duke Math. J. 170 (2021), no. 13, 2891--2956.
 
\bibitem[MT]{MT}
W.H.  Meeks III  and G. Tinaglia,  
{\it Geometry of constant mean curvature surfaces in $\RR^3$},
J. Eur. Math. Soc. (JEMS) 27 (2025), no. 6, 2467--2476. 
 
 \bibitem[N]{N}
 N. Nadirashvili, 
{\it Hadamard's and Calabi-Yau's conjectures on negatively curved and minimal surfaces},
Invent. Math. 126 (1996), no. 3, 457--465. 


\bibitem[Os]{Os}
R. Osserman, {\it A survey of minimal surfaces},   Dover, 2nd.
edition (1986).


\bibitem[P]{P}
J. P\'erez,  {\it A new golden age of minimal surfaces},
 Notices Amer. Math. Soc. 64 (2017), no. 4, 347--358. 


\bibitem[SS]{SS}
R. Schoen and L. Simon,  {\it Regularity of simply connected surfaces with quasiconformal Gauss map. Seminar on minimal submanifolds}, 127--145,
Ann. of Math. Stud., 103, Princeton Univ. Press, Princeton, NJ, 1983.


\bibitem[Ya]{Ya}
S.T. Yau, {\it Harmonic functions on complete Riemannian
manifolds}, Comm. Pure Appl. Math. 28 (1975) 201--228.


\end{thebibliography}
\end{document}